\documentclass[11pt]{amsart}
\usepackage[all]{xy}
\usepackage{array}
\usepackage{amsthm}
\usepackage{amsmath}
\usepackage{verbatim}
\theoremstyle{definition} \theoremstyle{plain}
\newtheorem{lemma}{Lemma}[section]

\newtheorem*{lemma*}{Lemma}
\newtheorem{theorem}{Theorem}[section]
\numberwithin{equation}{section}
\newcommand{\RHom}{\operatorname{RHom}}
\newcommand{\LDer}{\operatorname{\mathbb{L}}}
\newcommand{\Hom}{\operatorname{Hom}}

\title{Non-commutative Laurent phenomenon for two variables}
\begin{document}
\author{Alexandr Usnich}

\begin{abstract}
 We prove the non-commutative Laurent phenomenon for two variables
\end{abstract}

\maketitle

\section{Introduction}

	Let us consider an automorphism of the field $K=\mathbf{C}(x,y)$ given by the formula:
\[F:(x,y)\mapsto(\frac{H(x)}{y},x),\]
where $H(x)=1+h_1x+\dots +h_{n-1}x^{n-1}+x^n$ is a reversible polynomial, i.e. $h_i=h_{n-i}$.

The iterations of $F$ are actually given by Laurent polynomials \cite{FZ1}. It means that for any integer $k$ we have:
\[F^k:(x,y)\mapsto(L_1(x,y),L_2(x,y)),\]
	where $L_1,L_2\in \mathbf{C}[x,x^{-1},y,y^{-1}]$ are Laurent polynomials.
	
We introduce a non-commutative analog of this transfomation: consider
\[F_{nc}:(x,y)\mapsto (y^{-1}H(x),y^{-1}xy).\] 

We view $x,y$ as elements freely generating the non-commutative algebra $A$ by addition, multiplication and taking inverses of some elements. Namely, we have a ring morphism $\phi:A\rightarrow \mathbf{C}(x,y)$, and we can invert elements $a$ which don't belong to the kernel of $\phi$. Then $F_{nc}$ is an automorphism of the algebra $A$. If we allow to invert only elements $x,y$, then we will obtain the non-commutative subalgebra $\mathbf{C}<x,x^{-1},y,y^{-1}>\subset A$ which we call the ring of non-commutative Laurent polynomials.

We will prove the following result, conjectured by M.Kontsevich:
\begin{theorem}\label{main_thm}
		For any integer $k$ and for any reversible polynomial $H(x)$, the transformation $F_{nc}^k$ is given by non-commutative Laurent polynomials.
\end{theorem}

We call this \textit{the non-commutative Laurent phenomenon}.

Observe that multiplicative commutator $q=x^{-1}y^{-1}xy$ is preserved by $F_{nc}$. In the light of deformation quantization, people often consider algebra, where $q$ is a central element. We'd like to emphasize that we impose no such condition.

A special case of this transformation, where $H(x)=1+x^n$, turns up in the study of cluster mutations. In the article \cite{US0904} we prove the special case of the Laurent phenomenon for $n=2$ using explicite computations with matrices. In \cite{DiFK2} an alternative proof of the Laurent phenomenon for $n=2$ is given, via a combinatorial path-counting argument. It is moreover proved that the coefficients of Laurent polynomials are positive.

The main idea of our proof of Theorem \eqref{main_thm} is as follows. First we resolve birational map $F^k$, namely we construct a sequence of surfaces $Y_i$ and morphisms $\pi_i:Y_i\rightarrow\mathbb{P}^1\times\mathbb{P}^1$, such that the induced birational maps $F_i=\pi_{i+1}^{-1}\circ F\circ \pi_i:Y_i\rightarrow Y_{i+1}$ extend to natural biregular isomorphisms. Here $F$ is as defined previously in affine coordinates $x,y$ on $\mathbb{P}^1\times\mathbb{P}^1$. The surface $Y_i$ is constructed as a blow-up of a toric surface $Y_i^0$, which is a toric weighted blow-up of $\mathbb{P}^1\times\mathbb{P}^1$, at $2n$ points situated on the chain of toric divisors. We denote by $D_i$ the chain of rational curves on $Y_i$ which is the strict transform of toric divisors on $Y_i^0$. In fact the isomorphism $F_i$ sends the chain $D_i$ to the chain $D_{i+1}$.

Next, we construct quotient triangulated category
\[\widetilde{C}(Y_i)=\widetilde{D}(Y_i)/\widetilde{D}^1(Y_i),\]
 where $\widetilde{D}(Y_i)$ is a full subcategory of $D(Y_i)$ the derived category of coherent sheaves on $Y_i$ consisting of objects, which are left orthogonal to $O_{Y_i}$. $\widetilde{D}^1(Y_i)$ is a full subcategory of $\widetilde{D}(Y_i)$ consisting of objects which restrict to $0$ at the generic point. We use some properties of the category $\widetilde{C}(Y_i)$, which are proved in \cite{Us07}. Namely this category is generated\footnote{by shifts and taking cones} by one object $Q_i$, which is the image of the line bundle $\pi_i^*O(1,1)\in \widetilde{D}(Y_i)$. Moreover, we have:
 \[\Hom_{\widetilde{C}(Y_i)}(Q_i,Q_i)=A,\]
  where $A$ is the non-commutative algebra, containing distinguished elements $x,y$. The functor $\LDer F_i^*$ descends to an equivalence of quotient categories $\LDer F_i^*:\widetilde{C}(Y_{i+1})\rightarrow \widetilde{C}(Y_i)$. In \eqref{identification} we write down a specific isomorphism between $Q_i$ and $\LDer F_i^*Q_{i+1}$ in $\widetilde{C}(Y_i)$. This gives us an automorphism $F_{nc}$ of $A$, which doesn't depend on $i$:
  \begin{align*}
\xymatrix{F_{nc}:A=\Hom_{\widetilde{C}(Y_{i+1})}(Q_{i+1},Q_{i+1})\ar[r]^{\LDer F_i^*} & \Hom_{\widetilde{C}(Y_i)}(\LDer F_i^*Q_{i+1},\LDer F_i^*Q_{i+1})\ar[d]\\
  & \Hom_{\widetilde{C}(Y_i)}(Q_i,Q_i)=A.
}
  \end{align*}  

In the Lemma \ref{formula} we compute this automorphism explicitly:
	\[F_{nc}:(x,y)\mapsto (y^{-1}H(x),y^{-1}xy).\]

Therefore we see, that the functor 
\[\LDer \Phi^*=\LDer F_0^* \circ\dots\circ \LDer F_{k-1}^*: \widetilde{C}(Y_k)\rightarrow \widetilde{C}(Y_0)\]
together with the appropriate isomorphism of objects $\Phi^*(Q_k)$ and $Q_0$ in $\widetilde{C}(Y_0)$ induces an automorphism $F_{nc}^k$ of $A$:
\[A=\Hom_{\widetilde{C}(Y_k)}(Q_k,Q_k)\rightarrow \Hom_{\widetilde{C}(Y_0)}(Q_0,Q_0)=A.\]

Then we observe, that morphisms $F_{nc}^k(x),F_{nc}^k(y):Q_0\rightarrow Q_0$ descend from morphisms in the quotient category  $\widetilde{D}(Y_0)/\widetilde{D}_{D_0}(Y_0)$(observe $\widetilde{D}_{D_0}(Y_0)\subset\widetilde{D}^1(Y_0)$), where $\widetilde{D}_{D_0}(Y_0)$ is the subcategory of objects supported on the chain of rational curves $D_0$. Therefore $F_{nc}^k(x),F_{nc}^k(y)$ can be viewed also as elements in the endomorphism algebra of $\pi_0^*O(1,1)$ in the quotient category $\widetilde{D}(Y_0)/\widetilde{D}_B(Y_0)$, where $B$ is the union of $D_0$ and $2n$ exceptional curves of the blow-up of $Y_0^0$. 

Finally we prove in the Lemma \ref{nc_Laurent_quotient} that the image of the natural functor \[\Hom_{\widetilde{D}(Y_0)/\widetilde{D}_B(Y_0)}(Q_0,Q_0)\rightarrow \Hom_{\widetilde{C}(Y_0)}(Q_0,Q_0)=A\]
 is the subalgebra of non-commutative Laurent polynomials 
 \[\mathbf{C}<x,x^{-1},y,y^{-1}>\subset A.\]
 
 As we observed, $F_{nc}^k(x),F_{nc}^k(y)$ belong to this subalgebra, so they are non-commutative Laurent polynomials.
 
 I would like to thank M.Kontsevich, for initiating this direction of research, T.Logvinenko for careful reading and correction of the paper, and J.Ayoub for useful discussions.

\section{Resolution of automorphism}

	The results of this section appear in \cite{Us06} in greater generality. We summarize them here for the convenience of our reader.
	
Consider the birational automorphism of $\mathbb{P}^1\times\mathbb{P}^1$ given by the formula:
 \[F:(x,y)\mapsto(\frac{H(x)}{y},x),\]
 where $H(y)=1+h_1x+...h_{n-1}x^{n-1}+x^n$ is a polynomial of degree $n$. In the homogeneous coordinates $(X:Z)\times(Y:W)$ on $\mathbb{P}^1\times\mathbb{P}^1$ the affine coordinates are expressed as $x=\frac{X}{Z}$, $y=\frac{Y}{W}$.
 
	We are interested in constructing explicitly rational surfaces $Y_0,\dots,Y_k$ equipped with morphisms $\pi_i:Y_i\rightarrow \mathbb{P}^1\times\mathbb{P}^1$ and with biregular isomorphisms $F_i:Y_i\rightarrow Y_{i+1}$, such that the following diagrams commute:
\begin{align}\label{existence_resolution}
\xymatrix{ Y_i \ar[r]^{F_i} \ar[d]_{\pi_i} & Y_{i+1} \ar[d]^{\pi_{i+1}} \\
\mathbb{P}^1\times\mathbb{P}^1 \ar[r]_{F} & \mathbb{P}^1\times \mathbb{P}^1},
\end{align}  
  
	Let us define two series of vectors in $\mathbf{Z}^2$ by a recursive relation:
 \[ p_0=(0,1), p_1=(-1,0), p_{i+1}=np_i-p_{i-1};\]
 \[ t_0=(1,0), t_1=(0,-1), t_{i+1}=nt_i-t_{i-1}.\]

	Consider toric surfaces $Y_i^0$ given by the fan spanned by vectors:
\[\{p_i,\dots,p_0,t_0,t_1,t_2,\dots,t_{n+2-i}\}.\]

Surface $Y_i$ is constructed as a blow-up of the surface $Y_i^0$ in $2n$ points. Fans of surfaces $Y_i^0$ contain sub-fan $\{p_1,p_0,t_0,t_1\}$, which defines a surface $\mathbb{P}^1\times\mathbb{P}^1$, so they admit natural toric projections to it. We can actually think of them as weighted blow-ups of $\mathbb{P}^1\times\mathbb{P}^1$. We use standard notations $(x,y)$ for the coordinates on toric surfaces. Namely if a vector $(a,b)$ corresponds to a toric divisor, then rational function $\frac{x^b}{y^a}$ induces a canonical(up to taking an inverse) rational coordinate on this divisor. By a canonical coordinate on a divisor $D$ we will mean a rational function, which induces an isomorphism of $D$ with $\mathbb{P}^1$. On each surface $Y_i^0$ toric divisors form a chain of rational curves. Their strict transforms form a chain of rational curves on the blow-up $Y_i$ and the canonical rational coordinates lift from each curve to its strict transform.

The toric divisors corresponding to vectors $t_i$ and $p_i$ will be denoted $T_i$ and $P_i$ respectively. Let $x$ be the canonical coordinate on $P_0$, and $y$ the canonical coordinate on $T_0$. Note that intersection points with other toric divisors have coordinates $0$ and $\infty$. 

We begin with lemma, which shows how to resolve birational transformation: $(x,y)\mapsto (\frac{H(x)}{y},x)$. Let $Z_1^0$ be the toric surface corresponding to the fan: $\{p_1,p_0,t_0,t_1,t_2\}$, and let $Z_2^0$ be the toric surface corresponding to the fan $\{p_2,p_1,p_0,t_0,t_1\}$. 

\begin{align}
\xymatrix{ Z_1 \ar[r]^{G} \ar[d] & Z_2 \ar[d] \\
Z_1^0 \ar[r] \ar[d]_{r_1} & Z_2^0 \ar[d]^{r_2} \\
P^1\times P^1 \ar[r]_{F} & P^1\times P^1}
\end{align} 

The surface $Z_1$ is a blow-up of $Z_1^0$ in the $n$ points on curve $P_0$, where $H(x)=0$. The surface $Z_2$ is a blow-up of $Z_2^0$ in the $n$ points on $T_0$ where $H(y)=0$.

\begin{lemma}\label{resolution}
	For any reversible polynomial $H$ with distinct roots, the induced map $G$ is a regular isomorphism of surfaces $Z_1$, $Z_2$. Moreover it preserves the canonical coordinates on the chain of strict transforms of toric divisors.
\end{lemma}
\begin{proof}
	
	We denote by $C_{ab}$ the cone in $\mathbf{R}^2$ spanned by vectors $a,b$. Such cones correspond to toric points, and we have coordinates in the neighbourhoood of these points on $Z_1^0$. 
	
The coordinates near toric point $C_{p_1,p_0}$ on $Z_1^0$ are $(x^{-1},y)$;\\
near $C_{p_0,t_0}$ are $(x,y)$;\\
near $C_{t_0,t_1}$ are $(x,y^{-1})$;\\
near $C_{t_1,t_2}$ are $(x^{-1},x^ny^{-1})$;\\
and near $C_{t_2,p_1}$, which is a singular toric point, are $(x^{-1},y^{-1},x^{-n}y)$.

When we blow-up surface $Z_1^0$ at $n$ points on $P_0$, we pull-back coordinates near toric points to $Z_1$, so the coordinates near pull-back of $C_{p_1,p_0}$ on $Z_1$ are $(x^{-1},\frac{y}{H(x^{-1})})$;\\
near pull-back of $C_{p_0,t_0}$ are $(x,\frac{y}{H(x)})$. 

The coordinates near other pull-backs are the same as on $Z_1^0$.

Birational transformation $F$ of $P^1\times P^1$ lifts to a birational map $F^0:Z_1^0\rightarrow Z_2^0$. Under this map toric divisors $P_1,P_0,T_0,T_1,T_2$ go to divisors $P_2,P_1,P_0,T_0,T_1$ respectively. We now prove that this map is regular everywhere except at $n$ points on the divisor $P_0$ where $H(x)=0$. Because $H$ has distinct roots, all these points are different.

To avoid confusion, we denote by $(u,v)$ the rational coordinates on $Z_2^0$ and on $Z_2$, so that $G^*u=\frac{H(x)}{y}$,$G^*v=x$. The map $G$ sends the neighbourhood of the point $C_{p_2,p_1}$ to the neighbourhood of the point $C_{p_1,p_0}$:

\[G^*:\mathbf{C}[u^{-1}v^n,v^{-1}]\rightarrow \mathbf{C}[x^{-1},\frac{y}{H(x^{-1})}],\]
\[G^*(u^{-1}v^n,v^{-1})=(\frac{x^ny}{H(x)},x^{-1})=(\frac{y}{H(x^{-1})},x^{-1}).\]

It is an isomorphism of affine neighbourhoods. The canonical coordinate $u^{-1}v^n$ of $P_2$ on $Z_2$ goes to $\frac{y}{H(x^{-1})}$, which is equal to $y$ on $P_1$, because divisor $P_1$ is defined by $x^{-1}=0$. The canonical coordinate $v^{-1}$ on $P_1$ goes to the canonical coordinate $x^{-1}$ on $P_0$.

We do similar verifications for other pull-backs of toric points. For the neighbourhood of $G^*(C_{p_1,p_0})=C_{p_0,t_0}$ we have:

\[G^*:\mathbf{C}[u^{-1},v]\rightarrow \mathbf{C}[x,\frac{y}{H(x)}],\]
\[G^*(u^{-1},v)=(\frac{y}{H(x)},x).\]

It is again an isomorphism of affine neighbourhoods, and the canonical coordinate $u^{-1}$ on $P_0$ goes to $\frac{y}{H(x)}$, which is equal to $y$ on $T_0$, because $T_0$ is defined by $x=0$ in this neighbourhood.

For the neighbourhood of $G^*(C_{p_0,t_0})=C_{t_0,t_1}$ we have:

\[G^*:\mathbf{C}[\frac{u}{H(v)},v]\rightarrow \mathbf{C}[x,y^{-1}],\]
\[G^*(u,v)=(\frac{H(x)}{yH(x)},x)=(y^{-1},x).\]

It is an isomorphism of affine neighbourhoods. The canonical coordinate $v$ on $T_0$ goes to $x$ on $T_1$. 
For the neighbourhood of $G^*(C_{t_0,t_1})=C_{t_1,t_2}$ we have:

\[G^*:\mathbf{C}[\frac{u}{H(v^{-1})},v^{-1}]\rightarrow \mathbf{C}[x^{-1},x^ny^{-1}],\]
\[G^*(\frac{u}{H(v^{-1})},v^{-1})=(\frac{H(x)}{yH(x^{-1})},x^{-1})=(x^ny^{-1},x^{-1}).\]

It is an isomorphism of affine neighbourhoods. The canonical coordinate $u$ on $T_1$ goes to $\frac{H(x)}{y}$ on $T_2$, but $T_2$ is defined by $x^{-1}=0$, so we can write $\frac{H(x)}{y}=(x^ny^{-1})H(x^{-1})=x^ny^{-1}$. This proves that map $G$ preserves canonical coordinates on toric divisors. 

The four neighbourhoods that we considered provide the covering of $Z_1$ except at the point $C_{t_2,p_1}$ and at $n$ points, each lying on the exceptional curve of the blow-up. We verify, that at these points $G$ is also an isomorphism.  

For the neighbourhood of $G^*(C_{t_1,t_2})=C_{t_2,p_1}$ we have:

\[F^*:\mathbf{C}[u^{-1},v^{-1},uv^{-n}]\rightarrow \mathbf{C}[x^{-1},y^{-1},x^{-n}y],\]
\[F^*(u^{-1},v^{-1},uv^{-n})=(\frac{y}{H(x)},x^{-1},\frac{H(x)}{x^ny})=((x^{-n}y)H(x^{-1})^{-1},H(x^{-1})y^{-1}).\]

This map is well defined outside the divisor $H(x^{-1})=0$. The point $C_{t_2,p_1}$ doesn't belong to this divisor, so $G$ is well defined at this point.

If $\lambda$ is a root of polynomial $H$, then we have coordinates $(\frac{x-\lambda}{y},y)$ near the point on the exceptional curve, where we have to verify that $G$ is regular. The coordinates near the corresponding point on $Z_2$ are $(u,\frac{v^{-1}-\lambda^{-1}}{u})$. It is straitforward to see that $G^*$ defines an isomorphism of local rings.  

\end{proof}

Recall that we defined the toric surface $Y_i^0$ as given by the fan 
\[\{p_{i},\dots,p_1,p_0,t_0,t_1,\dots,t_{n+1-i}\}.\]

 Let's blow it up at $n$ points where $P_0$ intersects $H(x)=0$, and at $n$ points where $T_0$ intersects $H(y)=0$. Here $x$ and $y$ are the canonical coordinates on $P_0$ and $T_0$ respectively. The canonical coordinate are defined up to an inverse, so the polynomial $H$ needs to be reversible, for the blow up not to depend on the choice of a coordinate. Let us denote this blow-up by $Y_i$. Let $D_i\subset Y_i$ be the strict transform of toric divisors under this blow-up. 

As a corollary of Lemma \ref{resolution} we have:

\begin{lemma}\label{regular_auto}
 	If the polynomial $H$ has distinct roots and is reversible, then the map $F$ induces a regular automorphism $F_i$ between $Y_i$ and $Y_{i+1}$. Moreover $F_i(D_i)=D_{i+1}$.
\end{lemma}
\begin{proof}
 	Note that the surface $Y_i$ can be obtained from the surface $Z_1$ of previous lemma, by making two kinds of blow-ups. First, we perform weighted blow-ups to introduce toric divisors $P_{i+1},\dots,P_2,T_3,\dots,T_{n+1-i}$. Then we to blow-up $n$ points on the divisor $T_0$, defined by the equation $H(y)=0$. We blow-up $Z_2$ in a similar fashion to obtain $Y_{i+1}$. By Lemma \ref{resolution}, the map $F$ lifts to regular map $G$ from $Z_1$ to $Z_2$. Divisor $P_3$ on $Z_2$ is a weighted blow-up at the point $C_{t_1,p_2}$. The weights are determined using expression of the vector $p_3$ as the linear combination of $t_1$ and $p_2$. But this expression is the same as the expression of the vector $p_2$ as the linear combination of $t_2$ and $p_1$. So $G$ sends the weighted blow-up corresponding to $P_3$ to the weighted blow-up corresponding to $P_2$. The same argument works for other toric divisors $P_a,T_b$. The toric divisors $P_{i+1},\dots,P_2$ are then maped to $P_{i+2},\dots,P_3$, as well as $T_3,\dots,T_{n+1-i}$ are mapped to $T_2,\dots,T_{n-i}$. 
 	
 Also $n$ points on $T_0$ where $H(y)=0$ are mapped to $n$ points on $P_0$ where $H(x)=0$, because the canonical coordinates are preserved by $G$ by the previous lemma. Therefore, the blow-ups we do to $Z_1$ to produce $Y_i$ correspond under isomorphism $G$ precisely to the blow-ups we do to $Z_2$ to produce $Y_{i+1}$, and hence $G$ lifts to an isomorphism $F_i:Y_i\rightarrow Y_{i+1}$. The last statement of lemma is also clear.
 
\end{proof}

This lemma implies, that we have a regular isomorphism of surfaces:
\[\Phi=F_{k-1}\circ\dots\circ F_0:Y_0\rightarrow Y_k.\]

\section{DG-category associated to a rational surface}

	Let $D(Y_i)$ denote the bounded derived category of coherent sheaves on $Y_i$. By Lemma \ref{regular_auto} we have a functor 
	\[\LDer F_i^*: D(Y_{i+1})\xrightarrow{\sim} D(Y_i),\]
	which is an equivalence of triangulated categories.
	
	In \cite{Us07} we've introduced the notion of $\widetilde{D}(Y_i)$ full triangulated subcategory of $D(Y_i)$ which consists of objects $E$ for which $\RHom_{Y_i}(E,O_{Y_i})=0$. As
	\[\LDer F_i^*O_{Y_{i+1}}=O_{Y_i},\]
	$\LDer F_i^*$ restricts to an equivalence
\[\LDer F_i^*:\widetilde{D}(Y_{i+1})\rightarrow \widetilde{D}(Y_i).\]

	Let $\pi_i:Y_i\rightarrow\mathbb{P}^1\times\mathbb{P}^1$ be the natural projections. Recall that $\widetilde{D}(\mathbb{P}^1\times\mathbb{P}^1)$ is generated by three objects: 
	\[\widetilde{D}(\mathbb{P}^1\times\mathbb{P}^1)=<O(1,0),O(0,1),O(1,1)>.\]
	
	Denote by $Q_i=\pi_i^*O(1,1)\in \widetilde{D}(Y_i)$ the pull-back of the line bundle $O(1,1)$ by $\pi_i$. 
	
	Let $D^1(Y_i)$ be the full subcategory of $D(Y_i)$ consisting of objects whose support is at most a divisor, in other words we take objects of $D(Y_i)$ which restrict to $0$ at the generic point of $Y_i$. Let $D^1_{D_i}(Y_i)$ be the subcategory of $D^1(Y_i)$ consisting of objects supported on $D_i$, the union of all divisors $T_a$ and $P_b$.
	
	Observe that $\LDer F_i^*$ takes subcategories $D_{D_{i+1}}^1(Y_{i+1})$ and $D^1(Y_{i+1})$ to subcategories $D_{D_i}^1(Y_i)$ and $D^1(Y_i)$ respectively. This is because $F_i$ is a regular isomorphism, and $F(D_i)=D_{i+1}$. Let also:
	
\[\widetilde{D}^1(Y_i)=D^1(Y_i)\cap\widetilde{D}(Y_i),\]
\[\widetilde{D}_{D_i}(Y_i)=D_{D_i}^1(Y_i)\cap\widetilde{D}(Y_i).\]

	The non-commutative cluster mutations appear, when we look at the factor category 
\[\widetilde{C}(Y_i)=\widetilde{D}(Y_i)/\widetilde{D}^1(Y_i).\]

	It is proved in \cite{Us07}, that this category is a birational invariant of a variety. For rational surfaces it is generated by one object $Q_i$, and moreover
	\[\Hom_{\widetilde{C}(Y_i)}(Q_i,Q_i)=A,\]
	where $A$ is a non-commutative algebra. This algebra is a natural setting for non-commutative cluster mutations. Let us recall some properties of this algebra $A$. First of all there is an embedding $i:\mathbf{C}<x,y>\hookrightarrow A$, and there is a natural map $\phi:A\rightarrow\mathbf{C}(x,y)$. Moreover the kernel of the map $\phi$ is a commutator ideal of $A$:
	\[\ker(\phi)=A[A,A].\]
	We also have the following property: any $a\in A$ with $\phi(a)\ne 0$ is invertible.
	
	We now choose a way to identify objects $Q_i$ and $F^*Q_{i+1}$ in $\widetilde{C}(Y_i)$. This will induce a map on endomorphism ring of object, so we will get a map $F_{nc}:A\rightarrow A$, which we will compute explicitely.
	
Recall, that $F$ induces a regular map from $Z_1$ to $Z_2$, and we lift it after making some blow-ups to a regular map from $Y_i$ to $Y_{i+1}$. So we can choose an identification of $O_{Z_1}(1,1)$ and $G^*O_{Z_2}(1,1)$ on $Z_1$ in $\widetilde{C}(Z_1)$, and then lift this identification to $\widetilde{C}(Y_i)$. 

	The surface $Z_2$ is the blow-up of toric surface $Z_2^0$ at $n$ points on the toric divisor $T_0$. We identify this divisor with its strict transform. Denote by $E$ the exceptional curve of this blow-up. It is the union of $n$ rational curves. Also $Z_2$ has a chain of rational curves $P_2,P_1,P_0,T_0,T_1$. And we have linear equivalences of divisors:
	 \[O_{Z_2}(0,1)=P_0=T_1+P_2,\]
	 \[O_{Z_2}(1,0)=T_0+E=P_1+nP_2.\] 

	The divisor $O_{Z_2}(1,1)$ is therefore linearly equivalent to $T_1+P_1+(n+1)P_2$. Then we compute its pull-back by $G$ to $Z_1$:
	\[G^*O_{Z_2}(1,1)=G^*(T_1+P_1+(n+1)P_2)=T_2+P_0+(n+1)P_1.\]
	
	On $Z_1$ we have a chain of rational curves $P_1,P_0,T_0,T_1,T_2$, and we have the exceptional curve $C$ of the blow-up of $P_0$ at $n$ points.	Note that the effective divisor $O_{Z_1}(1,1)(-C)=P_0+P_1+T_2$ is dominated by $G^*O_{Z_2}(1,1)=T_2+P_0+(n+1)P_1$, so we have a natural morphism of line bundles 
	\[i_0:O_{Z_1}(1,1)(-C)\xrightarrow{nP_1} G^*O_{Z_2}(1,1).\]
	
	 There is also a unique up to scalar multiplication map of line bundles $i_1:O_{Z_1}(1,0)\xrightarrow{P_0} O_{Z_1}(1,1)(-C)$, which lifts the map $O(1,0)\rightarrow O(1,1)$ on $\mathbb{P}^1\times\mathbb{P}^1$, which vanishes along the divisor $P_0$. Finally there is a map $i_3:O_{Z_1}(1,0)\xrightarrow{T_1+nT_2} O_{Z_1}(1,1)$, which vanishes along the divisor $T_1$. All in all, we have the following sequence of maps of line bundles on $Z_1$:
\begin{align}\label{identification}	
	G^*O_{Z_2}(1,1)\xleftarrow{i_1} O_{Z_1}(1,1)(-C)\xleftarrow{i_2} O_{Z_1}(1,0)\xrightarrow{i_3} O_{Z_1}(1,1).
\end{align}

	If we consider the line bundles in this diagram as objects of the derived category of coherent sheaves $D(Z_1)$ then they belong to $\widetilde{D}(Z_1)$. If we pull \eqref{identification} back to $Y_i$, the objects will belong to $\widetilde{D}(Y_i)$. We now claim that the cones of the morphisms in \eqref{identification} belong to $D_{D_i}^1(Y_i)$. Indeed, on the surface $Z_1$ the object $Cone(i_1)$ is supported on $P_1$, $Cone(i_2)$ is supported on $P_0$, $Cone(i_3)$ is supported on $T_1\cup T_2$. Blowing up $n$ points on the curve $T_0$ doesn't change this. Therefore on $Y_i$ all these cones belong to $D_{D_i}(Y_i)$. Consequently, the morphisms in \eqref{identification} become isomorphisms in the quotient categories $\widetilde{D}(Y_i)/\widetilde{D}_{D_i}(Y_i)$ and $\widetilde{C}(Y_i)=\widetilde{D}(Y_i)/\widetilde{D}^1(Y_i)$. In the later category we thus obtain a particular isomorphism $j_k:Q_i\rightarrow F_i^*Q_{i+1}$. This isomorphism allows us to define an automorphism of the ring $A$:
\begin{align}\label{nc_definition}
\xymatrix{
	F_{nc}:A=\Hom_{\widetilde{C}(Y_{i+1})}(Q_{i+1},Q_{i+1})\ar[r]^{\LDer F_i^*} & \Hom_{\widetilde{C}(Y_i)}(F_i^*Q_{i+1},F_i^*Q_{i+1}) \ar[d]^{j_k^*}\\
	 & A=\Hom_{\widetilde{C}(Y_i)}(Q_i,Q_i).
}
\end{align}

\begin{lemma}\label{formula}
	The map $F_{nc}$ is given by
	\[F_{nc}:(x,y)\mapsto (y^{-1}H(x),y^{-1}xy)\]
\end{lemma}
\begin{proof}
	We first explain, how we identify the algebra $A$ with the endomorphism ring $\Hom_{\widetilde{C}(Y_i)}(Q_i,Q_i)$. If $\pi:X\rightarrow Y$ is a blow-up of a surface $Y$ at the smooth point, then $\LDer \pi^*$ induces a fully faithful embedding, and we have a semiorthogonal decomposition:
	\[D(X)=<O_E,\LDer\pi^*D(Y)>,\]   
	where $O_E$ is a structure sheaf of the exceptional curve $E$ of the blow-up. The similar decomposition works for weighted blow-ups. Functor $\LDer\pi^*$ then induces equivalences between quotient categories $\widetilde{C}(X)$ and $\widetilde{C}(Y)$. In particular, for a surface $Y_i$ we use sequence of blow-ups $\pi_i:Y_i\rightarrow \mathbb{P}^1\times\mathbb{P}^1$ to identify $\widetilde{C}(Y_i)$ with $\widetilde{C}=\widetilde{C}(\mathbb{P}^1\times\mathbb{P}^1)$. 
	
	If $(X:Z)\times(Y:W)$ are homogeneous coordinates on $\mathbb{P}^1\times\mathbb{P}^1$, then diagram 
	\begin{align}\label{def_x}
		O(1,1)\xleftarrow{Z} O(0,1)\xrightarrow{X} O(1,1)
	\end{align}
		defines an element in $\Hom_{\widetilde{C}}(O(1,1),O(1,1))$, which we denote by $x$. Similarly the diagram 
	\begin{align*}\label{def_y}
			O(1,1)\xleftarrow{W} O(1,0)\xrightarrow{Y} & O(1,1)
	\end{align*}
	defines an element in $\Hom_{\widetilde{C}}(O(1,1),O(1,1))$, which we denote by $y$.
		
	In the article \cite{Us07} we computed, that $\Hom_{\widetilde{C}}(\mathbb{P}^2)(O(2))=A$. If $(X:Y:Z)$ are homogeneous coordinates on $\mathbb{P}^2$, then denote by $x,y$ the elements represented by diagrams $O_{\mathbb{P}^2}(2)\xleftarrow{Z} O_{\mathbb{P}^2}(1)\xrightarrow{X} O_{\mathbb{P}^2}(2)$ and $O_{\mathbb{P}^2}(2)\xleftarrow{Z} O_{\mathbb{P}^2}(1)\xrightarrow{Y} O_{\mathbb{P}^2}(2)$ respectively. Consider the toric surface $T$, given by the fan $(1,0),(0,-1),(-1,-1),(-1,0),(0,1)$. It admits toric projections to both $\mathbb{P}^2$ and $\mathbb{P}^1\times\mathbb{P}^1$. We can therefore pull-back both $D(\mathbb{P}^2)$ and $D(\mathbb{P}^1\times\mathbb{P}^1)$ to $D(T)$ and compare the diagrams there. We observe that on surface $T$ the divisors $O_T(1,0),O_T(0,1)$ embed into $O_T(1)$, and the divisor $O_T(1,1)$ embeds into $O_T(2)$. Moreover, the diagrams that define $x,y$ in $\widetilde{C}(\mathbb{P}^1\times\widetilde{P}^1)$ and $\widetilde{C}(\mathbb{P}^2)$ give the same morphisms in $\widetilde{C}(T)$, thus identifying $\Hom_{\widetilde{C}(\mathbb{P}^1\times\mathbb{P}^1)}(O(1,1),O(1,1))$ with $A$.

	We now compute the action of $F_{nc}$ on $A$. It is enough to compute the action on elements $x,y$. First we compute the preimages of line bundles on $Z_1$:
	\[G^*O_{Z_2}(1,0)=G^*(P_1+nP_2)=P_0+nP_1,\]
	\[G^*O_{Z_2}(0,1)=G^*(T_1+P_2)=P_1+T_2,\]
	\[G^*O_{Z_2}(1,1)=G^*(P_1+nP_2+T_1)=P_0+(n+1)P_1+T_2.\]

Next recall that to represent $x,y$ on $Z_2$ we need the following maps:
\[X,Z:O_{Z_2}(0,1)\rightarrow O_{Z_2}(1,1),\] 
\[Y,W:O_{Z_2}(1,0)\rightarrow O_{Z_2}(1,1).\]

The map $X$ defines an inclusion of line bundles $O_{Z_2}(0,1)\xrightarrow{T_0+E} O_{Z_2}(1,1)$, given by divisor $T_0+E$. Consequently we will write the equality, where both sides are understood as inclusions of line bundles:
	\[X=T_0+E.\]
	
In a similar way we compute:
\[Z=P_1+nP_2,\]
\[Y=P_0,\]
\[W=T_1+P_2.\]

	Therefore we can compute the pull-backs:
\[G^*X=T_1+G^*(E),\]
\[G^*Z=P_0+nP_1,\]
\[G^*Y=T_0,\]
\[G^*W=T_2+P_1.\]

 Let $\pi_1'$ be the projection of $Z_1$ to the toric surface $Z_1^0$, and $C$ is the exceptional curve of the blow-up. As before $\pi_1$ is the toric projection from $Z_1^0$ to $\mathbb{P}^1\times\mathbb{P}^1$. We can write, using the same notation for the toric divisor $P_0$ on $Z_1^0$ and for its strict transform to $Z_1$:
	\[P_0+C=\pi_1'^*P_0.\]
	
	Moreover for any line bundle $L$ on $Z_1^0$ we have
	\[(\pi_1'^*L)_{C}=O_{C}.\]  
	
	This implies that there are exact sequences of coherent sheaves on $Z_1$:
	\[0\rightarrow G^*O_{Z_2}(1,0)\rightarrow (\pi_1')^*(P_0+nP_1)\rightarrow O_{C}\rightarrow 0,\]
	\[0\rightarrow G^*O_{Z_2}(1,1)\rightarrow (\pi_1')^*(P_0+(n+1)P_1+T_2)\rightarrow O_{C}\rightarrow 0.\]
		
	Observe that $O_C$ is an object of $\widetilde{D}(Z_1)$. The terms on the left and on the right in both sequences belong to the category $\widetilde{D}(Z_1)$, therefore so do the terms in the middle. 
	
	On the surface $Z_1^0$ we have:
	\[O_{Z_1^0}(1,0)=P_1+T_2.\]
	
	Therefore we have inclusions of line bundles on $Z_1^0$, which are isomorphisms outside $T_2$:
	\[P_0+nP_1\xrightarrow{nT_2} P_0+nP_1+nT_2=O_{Z_1}(n,1)),\]
	\[P_0+(n+1)P_1+T_2\xrightarrow{nT_2} P_0+(n+1)P_1+(n+1)T_2=O_{Z_1}(n+1,1).\]
	
	We use objects $O(n,1),O(n+1,1)\in\widetilde{D}(\mathbb{P}^1\times\mathbb{P}^1)$, and their pullbacks $O_{Z_1}(n,1),O_{Z_1}(n+1,1)$ to $Z_1$. We have inclusions of line bundles on surface $Z_1$:
	\[i:G^*O_{Z_2}(1,0)\xrightarrow{C+nT_2} O(n,1),\]
	\[j:G^*O_{Z_2}(1,1)\xrightarrow{C+nT_2} O(n+1,1).\]
	
And therefore we have compositions
	\[j\circ \LDer G^*X,j\circ \LDer G^*Z:G^*O_{Z_2}(0,1)=O_{Z_1}(1,0)\rightarrow O_{Z_1}(n+1,1).\]

	Now recall that $G^*X=T_1+G^*E$, $j=C+nT_2$, therefore $j\circ F^*X=C+T_1+G^*E+nT_2$. But the morphism $j\circ G^*X$ is a lift of a morphism from $\mathbb{P}^1\times\mathbb{P}^1$, where it is given by 
	\[(\pi_1)\circ(\pi_1')(C+T_1+G^*E+nT_2)=T_1+\pi_1(\pi_1'(G^*E)).\]

The divisor $\pi_1(\pi_1'(G^*E))$ is given by $H(x)=0$ on $\mathbb{P}^1\times\mathbb{P}^1$. If we introduce homogeneous polynomial $H(X,Z)$ defined by the condition that $\frac{H(X,Z)}{Z^n}=H(\frac{X}{Z})$, then the inclusion of line bundles $j\circ G^*X:O_{Z_1}(1,0)\rightarrow O_{Z_1}(n+1,1)$ is a pullback of the map $H(X,Z)W$. We can write 
\[j\circ G^*X=Z^n H(\frac{X}{Z})W.\]

 By the similar argument $j\circ G^*Z=(\pi_1\circ\pi_1')^*(P_0+nP_1)$. It follows that $j\circ G^*Z=Z^nY$.

	Both inclusions $i$ and $j$ are given by the same divisor $C+nT_2$, so 
	\[G^*Y,G^*W\in \Hom(G^*O_{Z_2}(1,0),G^*O_{Z_2}(1,1))=\Hom(O(n,1),O(n+1,1)).\]
	
Inclusion $G^*Y$ is given by $T_0$, so $j\circ G^*Y=X\circ i$. Inclusion $G^*W$ is given by $T_2+P_1$, so $j\circ G^*W=Z\circ i$.

	We also need to know the map $j\circ i_1\circ i_2:O(1,0)\rightarrow O(n+1,1)$, where $i_1,i_2$ are used in \eqref{identification} to identify $O_{Z_1}(1,0)$ and $G^*O_{Z_2}(1,1)$ in $\widetilde{C}(Z_1)$.
	
	 In our notations $i_3=W$. By using the similar techniques, we see that on the surface $\mathbb{P}^1\times\mathbb{P}^1$ we have $j\circ i_1\circ i_2=(\pi_1\circ\pi_1')^*(P_0+nP_1)$. It implies that $j\circ i_1\circ i_2=Z^nY$. 
	
	We use map $Z$ to identify $O_{Z_1}(l,1)$ and $O_{Z_1}(l+1,1)$ in $\widetilde{C}(Z_1)$, and the map $W$ to indentify $O_{Z_1}(1,l)$ and $O_{Z_1}(1,l+1)$ in $\widetilde{C}(Z_1)$. 
	
	 Let us denote by 
	\[\alpha\in\Hom_{\widetilde{C}(Z_1)}(O_{Z_1}(1,1),O_{Z_1}(n+1,1))\]
	the following morphism 
	\[\alpha=j\circ i_1\circ i_2\circ i_3^{-1}.\]
	
	Then we can write the element $F_{nc}(x)$ in the category $\widetilde{C}(Z_1)$ as:
\begin{align*}
F_{nc}(x)=\alpha^{-1}\circ j\circ F^*X\circ (F^*Z)^{-1}\circ j^{-1}\circ\alpha=\\
	=(Z^nY)^{-1}Z^n H(\frac{X}{Z})W(Z^nY)^{-1}(Z^nY)=y^{-1}H(x)y^{-1}y=y^{-1}H(x).
\end{align*}

	Similarly 
	\[F_{nc}(y)=(\alpha)^{-1}\circ j\circ F^*Y\circ (F^*W)^{-1}\circ j^{-1}\circ\alpha=(Z^nY)^{-1}XZ^{-1}(Z^nY)=y^{-1}xy.\]
	
	The claim of the lemma follows from the observation, that the pull-back along the map $Y_i\rightarrow Z_1$ induces equivalence of categories $\widetilde{C}(Z_1)$ and $\widetilde{C}(Y_i)$. In particular we note, that the formula for $F_{nc}$ doesn't depend on $i$.
\end{proof}

	We can now proceed to the final argument.
\begin{theorem}
		$F_{nc}^k(x),F_{nc}^k(y)$ are non-commutative Laurent polynomials.   
\end{theorem}
\begin{proof}
	First note that elements $F_{nc}^k(x),F_{nc}^k\in A$ are represented by elements of $\Hom_{\widetilde{C}(Y_0)}(Q_0,Q_0)$. Let $D_i$ be the chain of strict transforms of toric divisors from $Y_i^0$ to $Y_i$. We have natural functor
	\[\mathbb{K}_i:\widetilde{D}(Y_i)/\widetilde{D}_{D_i}(Y_i)\rightarrow\widetilde{D}(Y_i)/\widetilde{D}^1(Y_i)=\widetilde{C}(Y_i).\] 
	
In particular, we have the induced map
\[\mathbb{K}_i:\Hom_{\widetilde{D}(Y_i)/\widetilde{D}_{D_i}(Y_i)}(Q_i,Q_i)\rightarrow \Hom_{\widetilde{C}(Y_i)}(Q_i,Q_i).\]

\begin{lemma}	\label{intheimage} 
	 Elements $F_{nc}^k(x)$, $F_{nc}^k(y)$ belong to the image of $\mathbb{K}_0$. 
\end{lemma}
\begin{proof}

	By definition \eqref{nc_definition} of $F_{nc}$ we have:
	\[F_{nc}^k=j_0^*\circ \LDer F_0^*\circ\dots \circ j_{k-1}^*\circ \LDer F_{k-1}^*.\]
	
	If $\Phi=F_{k-1}\circ\dots \circ F_0:Y_0\rightarrow Y_k$, and $\sigma:\Phi^*O_{Y_k}(1,1)\xrightarrow{\sim} O_{Y_0}(1,1)$ is an appropriate identification in the category $\widetilde{C}(Y_0)$, then $F_{nc}^k$ is the composition
	\[\Hom_{\widetilde{C}(Y_k)}(Q_k,Q_k)\xrightarrow{\LDer \Phi^*} \Hom_{\widetilde{C}(Y_0)}(\Phi^*Q_k,\Phi^*Q_k) \xrightarrow{\sigma} \Hom_{\widetilde{C}(Y_0)}(Q_0,Q_0).\]
	
	Observe, that $x=X\circ Z^{-1}$ as defined in \eqref{def_x} is well-defined morphism in $\widetilde{D}(Y_k)/<Cone(Z)>$, because it uses the inverse of morphism $Z$. But $Supp(Cone(Z))=P_1\cup P_2\cup\dots\cup P_{k+1}\subset D_k$, so in particular it is an element of  $\widetilde{D}(Y_k)/\widetilde{D}_{D_k}(Y_k)$. Similarly $y=Y\circ W^{-1}$ is well-defined morphism of $\widetilde{D}(Y_k)/\widetilde{D}_{D_k}(Y_k)$, because $Supp(Cone(W))=T_1\cup P_2\cup\dots \cup P_{k+1}\subset D_k$.
	
 Lemma \eqref{regular_auto} implies, that $\Phi^{-1}(D_k)=D_0$. So $\LDer \Phi^*(x)$, $\LDer\Phi^*(y)$ are well-defined in the category $\widetilde{D}(Y_0)/\widetilde{D}_{D_0}(Y_0)$. 
 
 Morphism $\sigma$ is a composition of morphisms of the kind $\LDer (F_{i-1}\dots F_0)^*\circ j_i$. Recall that $j_i$ is defined in \eqref{nc_definition} using identification $i_3\circ i_2^{-1}\circ i_1^{-1}$ of $F_i^*Q_{i+1}$ and $Q_i$ as in \eqref{identification}. Observe, that $i_1,i_2,i_3$ are invertible isomorphisms in $\widetilde{D}(Y_i)/\widetilde{D}_{D_i}(Y_i)$, therefore $\sigma$ is invertible in $\widetilde{D}(Y_0)/\widetilde{D}_{D_0}(Y_0)$. This proves the lemma. 
	
\end{proof}

	 Let us take a curve $B=\pi_0^{-1}(XYZW=0)\subset Y_0$, which is the preimage of all toric divisors on $\mathbb{P}^1\times\mathbb{P}^1$. It is the union of strict transform of toric divisors $D_0$ and $2n$ exceptional curves of blow-up of $Y_0^0$. Then we have:
	
\begin{lemma}\label{nc_Laurent_quotient}
	In the quotient category $C=\widetilde{D}(Y_0)/\widetilde{D}_{B}(Y_0)$ we have:
	\[\Hom_{C}(O(1,1),O(1,1))=\mathbf{C}<x,x^{-1},y,y^{-1}>.\]
\end{lemma}
\begin{proof}	 
	By construction $\pi_0$ is a composition of regular maps: $Y_0\rightarrow Y_0^0\rightarrow \mathbb{P}^1\times\mathbb{P}^1$, where first arrow is a blow-up at $2n$ distinct smooth points, and $Y_0^0$ is a toric surface. 
	For the blow-up $\pi:Y_0\rightarrow Y_0^0$ with exceptional divisor $E$ we have a semiorthogonal decomposition\cite{BonOr01}, \cite{BonKaprep}:
	\[\widetilde{D}(Y_0)=<\LDer\pi'^*(\widetilde{D}(Y_0^0)),O_{E}>.\]
	
	So we have an equivalence of categories 
	\[\widetilde{D}(Y_0)/\widetilde{D}_{B}(Y_0)\rightarrow\widetilde{D}(Y_0^0)/\widetilde{D}(Y_0^0)_{tor}.\]
	
	In the last formula $\widetilde{D}(Y_0^0)_{tor}$ is the full subcategory of objects supported on toric divisors. Because of the semiorthogonal decomposition of the blow-up, this quotient category is the same for any toric surface. Even though $Y_0^0$ is not smooth, we can consider a smooth toric surface $T$ with an toric morphism $f:T\rightarrow Y_0^0$, and we can speak about the quotient $\widetilde{D}(T)/\widetilde{D}(T)_{tor}$ instead. We didn't do it in order to avoid cumbersome formulas. 
	
	As a consequence we have:
\begin{align*}		C=\widetilde{D}(Y_0)/\widetilde{D}_B(Y_0)=\widetilde{D}(\mathbb{P}^1\times\mathbb{P}^1)/\widetilde{D}_{(XYZW=0)}(\mathbb{P}^1\times\mathbb{P}^1)=\\
	=<O(0,1),O(1,0),O(1,1)>/<Cone(X),Cone(Y),Cone(Z),Cone(W)>=\\
	=D\left(\mathbf{C}<x,y>-mod\right)/<Cone(x),Cone(y)>.
\end{align*}
	
	In the last category we have: 
	\[\Hom(O(1,1),O(1,1))=\mathbf{C}<x,x^{-1},y,y^{-1}>.\]
\end{proof}

	We have the following maps
	\[\Hom_{\widetilde{D}(Y_0)/\widetilde{D}_{D_0}(Y_0)}(Q_0,Q_0)\rightarrow\Hom_{\widetilde{D}(Y_0)/\widetilde{D}_B(Y_0)}(Q_0,Q_0)\rightarrow\Hom_{\widetilde{C}(Y_0)}(Q_0,Q_0)=A.\]
	
	Lemma \ref{intheimage} implies that $F_{nc}^k(x),F_{nc}^k(y)$ belong to the image of the composition of these maps. In particular, they belong to the image of the second map, which is the subalgebra $\mathbf{C}<x,x^{-1},y,y^{-1}>\subset A$ by Lemma \ref{nc_Laurent_quotient}. This proves the theorem.
\end{proof}

\bibliographystyle{plain}
\bibliography{references}

\end{document}